\documentclass[reqno,12pt]{amsart}
\usepackage{amsmath,amsthm,enumerate,cite,graphics,pstricks,amsfonts,latexsym,amsopn,verbatim,amscd,amssymb}
\usepackage{color}
\usepackage{psfrag}
\usepackage[all]{xy}

\theoremstyle{plain}
\newtheorem{thm}{Theorem}[section]

\theoremstyle{definition}

\newtheorem{rem}[thm]{Remark}

\DeclareMathOperator{\Cent}{Cent}

\begin{document} 

\title[On finite groups with exactly two non-abelian centralizers]{On finite groups with exactly two non-abelian centralizers} 

\author[S. J. Baishya  ]{Sekhar Jyoti Baishya} 
\address{S. J. Baishya, Department of Mathematics, Pandit Deendayal Upadhyaya Adarsha Mahavidyalaya, Behali, Biswanath-784184, Assam, India.}

\email{sekharnehu@yahoo.com}

\begin{abstract}
In this paper, we  characterize finite group $G$ with unique proper non-abelian element centralizer. This improves \cite[Theorem 1.1]{nab}. Among other results, we have proved that if $C(a)$ is the proper non-abelian element centralizer of $G$ for some $a \in G$, then $\frac{C(a)}{Z(G)}$ is the Fitting subgroup of $\frac{G}{Z(G)}$, $C(a)$ is the Fitting subgroup of $G$ and  $G' \in C(a)$, where $G'$ is the commutator subgroup of $G$. 
\end{abstract}

\subjclass[2010]{20D60, 20D99}
\keywords{Finite group, Centralizer, Partition of a group}
%\thanks{*This paper is a part of my Ph.D thesis.}
\maketitle

\section{Introduction} \label{S:intro}

Throughout this paper $G$ is a finite group with center $Z(G)$ and commutator subgroup $G'$. Given a group $G$, let $\Cent(G)$ denote the set of centralizers of $G$, i.e., $\Cent(G)=\lbrace C(x) \mid x \in G\rbrace $, where $C(x)$ is the centralizer of the element $x$ in $G$. The study of finite groups in terms of $|\Cent(G)|$, becomes an interesting research topic in last few years. Starting with Belcastro and Sherman \cite{ctc092} in 1994  many authors have been studied and characterised finite groups $G$ in terms of $\mid \Cent(G)\mid$. More information on this and related concepts may be found in  \cite{ed09, amiri2019, amiri20191, rostami, en09, ctc09, ctc091, ctc099, baishya, baishya1, baishya2,baishya5, zarrin0941, zarrin0942, non, con}.

Amiri and Rostami \cite{nab} in 2015 introduced the notion of  $nacent(G)$ which is the set of all non-abelian centralizers of $G$. Schmidt \cite {schmidt} characterized all groups $G$ with $\mid nacent(G)\mid=1$ which are called CA-groups. The authors in \cite{nab} initiated the study of finite groups with  $\mid nacent(G)\mid=2$ and proved the following result (\cite[Theorem 1.1]{nab}):

\begin{thm}\label{nabcentamiri} 
Let $G$ be a finite group such that $\mid nacent(G)\mid=2$. If $C(a)$ is a proper non-abelian centralizer for some $a \in G$, then one of the following assertions hold:
\begin{enumerate}
	\item  $\frac{G}{Z(G)}$ is a $p$-group for some prime $p$.
	\item  $C(a)$ is the Fitting subgroup of $G$ of prime index $p$, $p$ divides $|C(a)|$ and $\mid Cent(G) \mid= \mid Cent(C(a)) \mid+ j+1 $, where $j$ is the number of distinct centralizers $C(g)$ for $g \in G \setminus C(a)$.
	\item  $\frac{G}{Z(G)}$ is a Frobenius group with cyclic Frobenius complement $\frac{C(x)}{Z(G)}$ for some $x \in G$.
\end{enumerate}
\end{thm} 
 
In this paper, we revisit finite groups $G$ with $\mid nacent(G)\mid=2$ and improve this result. Among other results, we have also proved that  if $C(a)$ is the proper non-abelian element centralizer of $G$ then $\frac{C(a)}{Z(G)}$ is the Fitting subgroup of $\frac{G}{Z(G)}$, $C(a)$ is the Fitting subgroup of $G$ and $G' \in C(a)$.

 \section{The main results}

In this section, we prove the main results of the paper. We make the following Remark from \cite[Pp. 571--575]{zappa} which will be used in the sequel.

\begin{rem}\label{rem1}
A collection $\Pi$ of non-trivial subgroups of a group $G$ is called a partition if every non-trivial element of $G$ belongs to a unique subgroup in $\Pi$. If $\mid \Pi \mid=1$, the partition is said to be trivial. The subgroups in $\Pi$ are called components of $\Pi$. Following Miller, if $\Pi$ is a non-trivial partition of a non-abelian $p$ group $G$ ($p$ a prime), then all the elements of $G$ having order $>p$ belongs to the same component of $\Pi$.

A partition $\Pi$ of a group $G$ is said to be normal if $g^{-1}Xg \in \Pi$ for every $X \in \Pi$ and $g \in G$.  A non-trivial partition $\Pi$ of a group $G$ is said to be elementary if $G$ has a normal subgroup $K$ such that all cyclic subgroups which are not contained in $K$ have order $p$ ($p$ a prime) and are components of $\Pi$. All normal non-trivial partitions of a $p$ group of exponent $> p$ are elementary. 

A non-trivial partition $\Pi$ of a group $G$ is said to be non-simple if there exists a proper normal subgroup $N$ of $G$ such that for every component $X \in \Pi$, either $X \leq N$ or $X \cap N=1$. Let $G$ be a group and $\Pi$ a normal non-trivial partition. Suppose $\Pi$ is not a Frobenious partition and is non-simple. Then $G$ has a normal subgroup $K$ of index $p$ ($p$ a prime) in $G$ which is generated by all elements of $G$ having order $\neq p$. So $\Pi$ is elementary.

Let $G$ be a group and $p$ be a prime. We recall that the subgroup generated by all the elements of $G$ whose order is not $p$ is called the Hughes subgroup and denoted by $H_p(G)$.  The group $G$ is said to be a group of Hughes-Thompson type if $G$ is not a $p$ group and $H_p(G) \neq G$ for some prime $p$. In such a group we have $\mid G : H_p(G) \mid=p$ and $H_p(G)$ is nilpotent.
\end{rem}

We now determine the structure of finite groups $G$ with $\mid nacent(G)\mid=2$ which improves (\cite[Theorem 1.1]{nab}).

\begin{thm}\label{nabcent1}
Let $G$ be a finite group and $a \in G \setminus Z(G)$. Then $nacent(G)=\lbrace G, C(a) \rbrace$ if and only if one of the following assertions hold:
\begin{enumerate}
	\item  $\frac{G}{Z(G)}$ is a non-abelian $p$-group of exponent $>p$ ($p$ a prime), $\mid  \frac{G}{Z(G)}: H_p(\frac{G}{Z(G)})\mid=p$,  $H_p(\frac{G}{Z(G)})=\frac{C(a)}{Z(G)}$, $\mid \frac{C(x)}{Z(G)} \mid=p$ for any $x \in G \setminus C(a)$ and $C(a)$ is a CA-group.
	\item  $\frac{G}{Z(G)}$ is a group of Hughes-Thompson type, $H_p(\frac{G}{Z(G)})=\frac{C(a)}{Z(G)}$ ($p$ a prime), $\mid \frac{C(x)}{Z(G)} \mid=p$ for any $x \in G \setminus C(a)$ and $C(a)$ is a CA-group.
	\item  $\frac{G}{Z(G)}= \frac{C(a)}{Z(G)} \rtimes \frac{C(x)}{Z(G)}$ is a Frobenius group with  Frobenius Kernel $\frac{C(a)}{Z(G)}$, cyclic Frobenius Complement $\frac{C(x)}{Z(G)}$ for some $x \in G\setminus C(a)$ and $C(a)$ is a CA-group..
\end{enumerate}
\end{thm}

\begin{proof}
Let $G$ be a finite group such that $nacent(G)=\lbrace G, C(a) \rbrace$, $a \in G \setminus Z(G)$. Then clearly $C(a)$ is a CA-group.

Note that we have $C(s) \subseteq C(a)$ for any $s \in C(a)\setminus Z(G)$,  $C(a) \cap C(x)=Z(G)$ for any $x \in G \setminus C(a)$ and $C(x) \cap C(y)=Z(G)$ for any $x, y \in G \setminus C(a)$ with $C(x) \neq C(y)$. Hence $\Pi= \lbrace \frac{C(a)}{Z(G)}, \frac{C(x)}{Z(G)} \mid x \in G \setminus C(a)\rbrace$ is a non-trivial partition of $\frac{G}{Z(G)}$.  In the present scenario we have $(gZ(G))^{-1}\frac{C(x)}{Z(G)}gZ(G)=\frac{g^{-1}C(x)g}{Z(G)}=\frac{C(g^{-1}xg)}{Z(G)}$ for any $gZ(G) \in \frac{G}{Z(G)}$ and $\frac{C(a)}{Z(G)}\lhd \frac{G}{Z(G)}$. Therefore $(gZ(G))^{-1}XgZ(G) \in \Pi$ for every $X \in \Pi$ and $gZ(G) \in \frac{G}{Z(G)}$. Hence $\Pi$ is a normal non-simple partition of $\frac{G}{Z(G)}$.

In the present scenario, if $\Pi$ is a Frobenius partition of $\frac{G}{Z(G)}$, then 
$\frac{G}{Z(G)}= \frac{C(a)}{Z(G)} \rtimes \frac{C(x)}{Z(G)}$ is a Frobenius group with Frobenious Kernel $\frac{C(a)}{Z(G)}$ and cyclic Frobenius Complement $\frac{C(x)}{Z(G)}$ for some $x \in G\setminus C(a)$.

Next, suppose $\Pi$ is not a Frobenius partition. Then in view of Remark \ref{rem1}, $\frac{G}{Z(G)}$ has a normal subgroup of index $p$ ($p$ a prime) in $\frac{G}{Z(G)}$ which is generated  by all elements of $\frac{G}{Z(G)}$ having order $\neq p$.

In the present situation if  $\frac{G}{Z(G)}$ is not a $p$ group ($p$ a prime), then in view of Remark \ref{rem1}, $\frac{G}{Z(G)}$ is a group of Hughes-Thompson type and $\Pi$ is elementary. That is  $\frac{G}{Z(G)}$ has a normal subgroup $\frac{K}{Z(G)}$ such that all cyclic subgroups which are not contained in $\frac{K}{Z(G)}$ have order $p$ ($p$ a prime) and are components of $\Pi$. In the present scenario we have $\frac{K}{Z(G)}=H_p(\frac{G}{Z(G)})$. Therefore $\Pi$ has $\frac{\mid G \mid}{p}$ components of order $p$ and these are precisely $\frac{C(x)}{Z(G)}, x \in G \setminus C(a)$. Consequently, we have $H_p(\frac{G}{Z(G)})=\frac{C(a)}{Z(G)}$.

On the other hand, if $\frac{G}{Z(G)}$ is a $p$ group ($p$ a prime), then  in view of Remark \ref{rem1}, $\frac{G}{Z(G)}$ is  non-abelian of exponent $>p$ and $\mid  \frac{G}{Z(G)}: H_p(\frac{G}{Z(G)})\mid=p$. In the present situation by Remark \ref{rem1}, $\Pi$ is elementary. Therefore using Remark \ref{rem1} again, $H_p(\frac{G}{Z(G)})=\frac{C(a)}{Z(G)}$ and all cyclic subgroups which are not contained in $\frac{C(a)}{Z(G)}$ have order $p$ and are components of $\Pi$. Therefore $\Pi$ has $\frac{\mid G \mid}{p}$ components of order $p$ and these are precisely $\frac{C(x)}{Z(G)}, x \in G \setminus C(a)$. Consequently, we have $H_p(\frac{G}{Z(G)})=\frac{C(a)}{Z(G)}$. 

Conversely, suppose $G$ is a finite group such that one of (a), (b) or (c) holds. Then it is easy to see that $nacent(G)=\lbrace G, C(a) \rbrace$  for some $a \in G \setminus Z(G)$. 
\end{proof}

As an immediate consequence we have the following result. Recall that for a finite group $G$, the Fitting subgroup denoted by $F(G)$ is the largest normal nilpotent subgroup of $G$.

\begin{thm}\label{sjb1} 
Let $G$ be a finite group with a unique proper non-abelian centralizer $C(a)$ for some $a \in G$. Then we have
\begin{enumerate}
	\item $\mid \Cent(G) \mid= \mid \Cent(C(a)) \mid+ \frac{\mid G \mid}{p}+1$, ($p$ a prime) or $\mid \Cent(C(a)) \mid+ \mid \frac{C(a)}{Z(G)} \mid+1$.
	\item $G' \subseteq C(a)$.
	\item  $\frac{C(a)}{Z(G)}$ is the Fitting subgroup of $\frac{G}{Z(G)}$.
	\item $C(a)$ is the Fitting subgroup of $G$.
	\item $C(a)=P \times A$, where $A$ is an abelian subgroup and $P$ is a CA-group of prime power order.
	\item $\frac{G}{C(a)}$ is cyclic.
\end{enumerate}
\end{thm}

\begin{proof}
(a) In view of Theorem \ref{nabcent1},
if $\frac{G}{Z(G)}$ is a Frobenius group then by \cite[Proposition 3.1]{amiri2019}, we have $\mid \Cent(G) \mid= \mid \Cent(C(a)) \mid+ \mid \frac{C(a)}{Z(G)} \mid+1$.

On the otherhand, if $\frac{G}{Z(G)}$ is not a Frobenius group, then it follows from the proof of Theorem \ref{nabcent1} that the non-trivial partition $\Pi= \lbrace \frac{C(a)}{Z(G)}, \frac{C(x)}{Z(G)} \mid x \in G \setminus C(a)\rbrace$ of $\frac{G}{Z(G)}$ has $\frac{\mid G \mid}{p}$ components of order $p$ and these are precisely $\frac{C(x)}{Z(G)}, x \in G \setminus C(a)$. Hence $\mid \Cent(G) \mid= \mid \Cent(C(a)) \mid+ \frac{\mid G \mid}{p}+1$.

(b) If $\frac{G}{Z(G)}$ is a Frobenius group, then by  Theorem \ref{nabcent1}, $\frac{G}{Z(G)}= \frac{C(a)}{Z(G)} \rtimes \frac{C(x)}{Z(G)}$ with cyclic Frobenius complement $\frac{C(x)}{Z(G)}$ for some $x \in G\setminus C(a)$. Therefore $\frac{G}{C(a)}$ is cyclic and hence $G' \subseteq C(a)$.

On the otherhand, if $\frac{G}{Z(G)}$ is not a Frobenius group, then it follows from Theorem \ref{nabcent1} that  $C(a) \lhd G$ and $\mid \frac{G}{C(a)} \mid=p$ ($p$ a prime). Hence $G' \subseteq C(a)$. 

(c) If $\frac{G}{Z(G)}$ is a Frobenius group, then by  Theorem \ref{nabcent1}, $\frac{G}{Z(G)}= \frac{C(a)}{Z(G)} \rtimes \frac{C(x)}{Z(G)}$ with cyclic Frobenius complement $\frac{C(x)}{Z(G)}$ for some $x \in G\setminus C(a)$. In the present scenario by \cite[Pp. 3]{mukti}, we have $\frac{C(a)}{Z(G)}$ is the Fitting subgroup of $\frac{G}{Z(G)}$.

On the otherhand, if $\frac{G}{Z(G)}$ is not a Frobenius group, then it follows from Theorem \ref{nabcent1} that    $H_p(\frac{G}{Z(G)})=\frac{C(a)}{Z(G)}$ and $\mid  \frac{G}{Z(G)}: H_p(\frac{G}{Z(G)})\mid=p$.   Hence $\frac{C(a)}{Z(G)}$ is the Fitting subgroup of $\frac{G}{Z(G)}$, noting that we have $C(a) \lhd G$. 

(d)It follows from (c) noting that $F(\frac{G}{Z(G)})=\frac{F(G)}{Z(G)}=\frac{C(a)}{Z(G)}$.

(e) Using (d) we have $C(a)$ is a nilpotent CA-group. Therefore  using \cite[Theorem 3.10 (5)]{abc}, $C(a)=P \times A$, where $A$ is an abelian subgroup and $P$ is a CA-group of prime power order.

(f) It is clear from Theorem \ref{nabcent1} that if $\frac{G}{Z(G)}$ is a Frobenius group, then $\frac{G}{C(a)}$ is cyclic and $\mid \frac{G}{C(a)} \mid=p$, ($p$ a prime) otherwise.
\end{proof}

%\section*{Acknowledgment}

\end{document}